\documentclass {amsart}

\usepackage{amsthm,amsmath,amssymb,verbatim} 
\usepackage[all]{xy}
\UseComputerModernTips

\theoremstyle{plain} 
	\newtheorem{thm}{Theorem}[section] 
	\newtheorem{prop}[thm]{Proposition} 
	\newtheorem{lemma}[thm]{Lemma}  
	\newtheorem{cor}[thm]{Corollary} 
 
\theoremstyle{definition}

	\newtheorem{examples}[thm]{Example}

\newcommand{\cate}{\mathcal E}

\newcommand{\aro}{\longrightarrow}

\newcommand{\ep}{\varepsilon}
\newcommand{\id}{\text{id}}
\newcommand{\N}{\mathbb N}
\newcommand{\R}{\mathbb R}
\newcommand{\h}{\mathbb H}

\newcommand{\catl}{\mathcal L}
\newcommand{\f}{ \varphi}
\newcommand{\F}{\Phi}
\newcommand{\Lip}{\text{\rm{Lip}}}
\newcommand{\ev}{\text{\rm{ev}}}
\newcommand{\diam}{\text{\rm{diam}}}
\newcommand{\card}{\text{\rm{card}}}

\newcommand{\map}{\text{\rm{map}}}

\begin{document} 

\title{Bi-Lipschitz embeddings revisited} 
 
\author[H. Movahedi-Lankarani]
	{H. Movahedi-Lankarani}
\address{Department of Mathematics\\
	Penn State Altoona\\
         Altoona, PA 16601-3760} 
\email{hxm9@psu.edu}
 
\author [R. Wells]
        {R. Wells}
\address{Department of Mathematics\\
         Penn State University\\
         University Park, PA 16802}


\date{\today}

\maketitle 

\abstract
     Given a metric space $(X, d)$, we continue our study of the distance function $x \mapsto d (x, -)$ and its relation to  bi-Lipschitz embeddings of $(X, d)$ into  $\R^N$.  As application, given a compact metric-measure space  $(X, d, \mu)$, we give three sufficient sonditions for the existence of such a bi-Lipschitz embedding.    
 \endabstract 



\section{Introduction}

A map $f: (X, d_X) \aro (Y, d_Y)$ of metric spaces is called Lipschitz if there exists $L > 0$ such that $d_Y \bigl(f (x), f (x^{\prime})\bigr) \le L\, d_X (x, x^{\prime})$  for all $x, x^{\prime} \in X$.  The infimum of such $L$ is called the Lipschitz constant of $f$ and is denoted $\Lip (f)$.  The map $f$ is called bi-Lipschitz provided that $f$ is invertible with $f^{- 1}$ also Lipschitz.   Equivalently, $f$ is bi-Lipschitz provided that there exist $0 < \ell \le L$ such that  
\[
\ell \, d_X (x, x^{\prime})  \le d_Y \bigl(f (x), f (x^{\prime})\bigr) \le L \, d_X (x, x^{\prime}),
\]

\noindent  for all $x, x^{\prime} \in X$.   In this case, the supremum of such $\ell$ is called the lower Lipschitz constant of $f$.  We are interested in the still open problem of determining when a metric space admits a bi-Lipschitz embedding into a finite dimensional Euclidean  space.

\vskip5pt

Our interest in bi-Lipschitz embeddings arises from the fact that isometric embeddings of metric spaces can be  very  rigid.  For instance, using a very simple argument, it is shown in \cite[Proposition 3.1]{LU-ML} that  no     infinite ultrametric   space admits an isometric embedding into a finite dimensional Euclidean space.  (We recall that an ultrametric $d$ on a set $X$ is a metric that satisfies the following strong triangle inequality:  $d (x, z)  \le \max \{d (x, y) , \, d (y, z)\}$, for all $x, y, z \in X$.)   Indeed, it is shown in \cite{aschbacher} that if an ultrametric space admits an isometric embedding into $\R^N$,  then it has at most $N + 1$ points.  In contrast, every doubling   ultrametric space   admits a bi-Lipschitz embedding into a  finite dimensional Euclidean space; see  \cite{assouad2,  LU-ML} and references therein.  Even better (or worse), there is a finite metric space which does not admint an isometric embedding into $\R^N$, any $N \ge 1$; see Example \ref{noiso}.  Obviously, there is a bi-Lipschitz embedding of any finite metric space into $\R^1$.

\vskip5pt

Now let $(X, d, \mu)$ be a metric-measure space.  That is, $(X, d)$ is a  metric space and $\mu$ is a   Borel regular measure on  $X$ which is non-trivial on nonempty open sets.  We assume, throughout,  that  $0 < \mu (X) < \infty$.   We  consider the canonical map $\iota_d: (X, d) \aro L^p (\mu)$, $1 \le p < \infty$, given by setting  $\iota_d: x \mapsto d (x, -)$.  This canonical map  is a one-one contraction, but not necessarily bi-Lipschitz.   Our interest lies in the question of when this canonical map is lower Lipschitz and its relationship to bi-Lipschitz embeddings  into finite dimensions \cite{bilip}.    For simplicity, we restrict to the case of the canonical map $\iota_d: (X, d) \aro L^2 (\mu)$.  

\vskip5pt

 In sections 2 and 3, we   first study the set  $\iota_d (X) \subset L^2 (\mu)$ and some closely related maps.  We then prove  three  bi-Lipschitz embedding results: Theorem \ref{iotarho}, Theorem \ref{iotaneariso}, and Theorem  \ref{d*d}.  We  finish by  remarking on some relevant issues.   This paper is a continuation of the work in \cite{bilip} so that, as an appendix, an erratum for \cite{bilip} is also included.


\section{The  canonical map  and radial projection}

Let $(X, d, \mu)$ be a metric-measure space and let $\iota_d: X \aro L^2 (\mu)$ be the canonical map so that, for every $x \in X$, we have $0 < \Vert\iota_{\kappa} (x)\Vert_2 \le \diam (X) \, \mu^{1/2} (X)$.  It is clear that for $x, y \in X$ the inner product
$\langle\iota_{d} (x), \iota_{d} (y)\rangle    \ge 0$.   If for some $x, y \in X$ we have $\langle\iota_{d} (x), \iota_{d} (y)\rangle  =  0$, then there exist subsets $A, B \subset X$ such that $d (x, z) = 0$ for all $z \in A$,  $d (y, z) = 0$ for all $z \in B$, and $\mu (A \cup B) = \mu (X) > 0$.   But $d$ is a metric so that $A = \{x\}$ and $B = \{y\}$.  Hence, since $\mu$ is regular Borel,  if $\card (X) \ge 3$,  then $\langle\iota_{d} (x), \iota_{d} (y)\rangle  >  0$ for all $x, y \in X$.  Also, the function $(x, y) \mapsto \langle\iota_{d} (x), \iota_{d} (y)\rangle$ is continuous and strictly positive.  Consequently, if $(X, d)$ is compact, then this function attains its minimum $s > 0$, and we  have the following lemma.

\begin{lemma}
\label{dotpositive} 
Let $(X, d, \mu)$ be a compact metric-measure space  with $\card (X) \ge 3$ and let $\iota_{d}: X \aro L^2 (\mu)$ be the canonical map.  Then there exists $s > 0$ such that for all $x, y \in X$ we have $\langle\iota_{d} (x), \iota_{d} (y)\rangle  \ge s >  0$.
\end{lemma}

 \noindent In this case, of course,  we have 
\[
0 < s \le \langle\iota_{d} (x), \iota_{d} (x)\rangle = \Vert\iota_{d} (x)\Vert_2^2 \le \diam (X) \Vert\iota_{d^{1/2}} (x)\Vert_2^2  \le \diam^2 (X) \mu (X).
\]

Let $S (\mu)$ denote the unit sphere of $L^2 (\mu)$ and let $\varpi: L^2 (\mu) \setminus \{0\} \aro S (\mu)$ be the radial projection $u \mapsto u/\Vert u\Vert_2$.  Being the gradient of the convex function $u \mapsto  \Vert u\Vert_2$, the map $\varpi$ satisfies $\langle\varpi (u) - \varpi (v), u - v\rangle$ for all $u, v \in L^2 (\mu)$.

\begin{lemma}
\label{radproj} 
Let $(X, d, \mu)$ be a  metric-measure space and let $\iota_{d}: X \aro L^2 (\mu)$ be the canonical map.  Then the radial projection $\varpi\big\vert_{\iota_{d} (X)}$ is injective.
\end{lemma}

\begin{proof}
If for some $x, y \in X$ we have $\iota_d (x) = r \iota_d (y)$ for some $r > 0$, then $d (x, z) = r \, d (y, z)$ for $\mu$-almost every $z \in X$.  Being a metric, $d$ is continuous and hence $d (x, z) = r \, d (y, z)$ for  every $z \in X$.  In particular, for $z =x$ we have $0 = d (x, x) = r \, d (y,  x)$, and $x = y$.  
\end{proof}

   The map $\iota_d$ induces a metric $\rho_d$ on $X$ by setting   $\rho_d (x, y)  = \Vert\iota_d (x) - \iota_d (y)\Vert_2$.     Also,   it follows from Lemma \ref{radproj} that $\iota_{d}$ induces  another metric   $\theta_{d}$ on $X$ by setting $\theta_{d} (x, y) =$ the arc-length on the unit sphere of $L^2 (\mu)$ from $\varpi \bigl(\iota_d (x)\bigl) = \iota_{d} (x) / \Vert\iota_{d} (x)\Vert_2$ to $\varpi \bigl(\iota_d (y)\bigl) = \iota_{d} (y) / \Vert\iota_{d} (y)\Vert_2$.   If $(X, d)$ is compact, then  
\[
0 \le \theta_{d} (x, y) \le \cos^{- 1} \left(\frac{s}{\Vert\iota_{d} (x)\Vert_2   \Vert\iota_{d} (y)\Vert_2}\right) < \pi/2.
\]

\begin{cor}
\label{maximal} 
Let $(X, d, \mu)$ be a compact metric-measure space  with $\card (X) \ge 3$ and let $\iota_{d}: X \aro L^2 (\mu)$ be the canonical map.  Then  there exist $x_{\circ}, y_{\circ} \in X$ such that $0 < \cos \theta_{d} (x_{\circ}, y_{\circ}) \le \cos\theta_{d} (x, y)$ for all $x, y \in X$.
\end{cor}

Of course there are  relationships among the three metrics $d$, $\rho_d$, and $\theta_{d}$.  For instance,  it is easily seen that $(2\, \sqrt s /\pi) \, \theta_{d} \le \rho_{d} \le d$.     However, $\theta_{d}$ and $\rho_{d}$ are not in general bi-Lipschitz equivalent. 
Letting 
\[
\kappa_{d} (x, y) =  \bigl\Vert \varpi \bigl(\iota_d (x)\bigl)  -  \varpi \bigl(\iota_d (y)\bigl)\bigr\Vert_2 = 2 \sin \left(\frac{\theta_{d} (x, y)}{2}\right),
\]

\noindent  it is clear that $(2/\pi) \theta_{d} \le \kappa_{d}  \le \theta_{d}$.       Interestingly,  the symmetric function 
\begin{equation}
\label{dd}
\Delta  (\kappa_{d}) (x, y) = \langle\kappa_d (x, -), \kappa_d (-, y)\rangle = \cos \theta_{d} (x, y)
\end{equation}

 \noindent does not satisfy the triangle inequality, but $1 -  \Delta (\kappa_{d})$ is  a pseudometric; see also Section 6.

\vskip5pt

  It follows from straightforward calculations that the radial projection $\varpi: L^2 (\mu) \setminus \{0\} \aro S (\mu)$ is in fact differentiable with 
\[
d \varpi (u) (v) = \frac{1}{\Vert u\Vert_2} \left[v - \frac{\langle u, v\rangle}{\Vert u\Vert_2^2} u\right].
\]

\noindent   This raises the following question:   Under what conditions is the map  $\varpi\big\vert_{\iota_{d} (X)}$  a $C^1$-embedding?   We will not address this question in this paper except to formulate the conjecture below.     We recall that  there is a notion of the tangent space at a point of an arbitrary subset of a Hilbert space introduced in \cite{MMLW}, and we have:

\vskip5pt

\noindent {\bf{Conjecture.}} {\em{ 
Let  $(X, d, \mu)$ be a compact metric-measure space and  let  $\iota_{d}: (X, d) \aro L^2 (\mu)$ be the canonical map.  Assume  that there exists  $0 < \ell \le 1$ satisfying
\begin{enumerate}
\item   $\ell \, d  (x,  y) \le \bigl\Vert \iota_{d} (x) - \iota_{d} (y)\bigr\Vert_2$ for all $x, y \in X$, and
\item   $\bigl\Vert \iota_{d^{1/2}} (x)\bigr\Vert_2^2  <  \ell \, \bigl\Vert \iota_{d} (x)\bigr\Vert_2$ for every $x \in X$.
\end{enumerate}

\noindent Then the radial projection $\varpi\big\vert_{\iota_{d} (X)}: \iota_{d} (X) \aro S (\mu)$ is a $C^1$-embedding.}}


 \section{Some related   maps}

Let $(X, d, \mu)$ be a metric-measure space.    The metric $d$ induces a linear transformation $T_{d}: L^2 (\mu) \aro L^2 (\mu)$ by setting 
\begin{equation}
\label{td}
T_{d}( f)  =  \langle\iota_{d} (-), f\rangle,
\end{equation}

\noindent for $ f \in L^2 (\mu)$.  Clearly $T_{d}$ is a Hilbert-Schmidt operator and hence compact and self-adjoint.   Moreover,  $T_d \bigl(\iota_d (x)\bigr) = \langle \iota_d (-), \iota_d (x)\rangle = \Delta (d)  (-, x) = \iota_{\Delta (d)} (x)$  so that we have the commutative diagram
 
\begin{equation}
\label{k*k}
\xymatrix{ & L^2 (\mu) \ar[d]^{T_d} \\
X \ar[ru]^{\iota_d} \ar[r]_{\iota_{\Delta (d)}} & L^2 (\mu).}
\end{equation}

We next consider a variant of the operator $T_{d}: L^2 (\mu) \aro L^2 (\mu)$.  To this end, we  write $\Lip \, (X)$ for the class of all real-valued Lipschitz functions on $(X, d)$; it is a normed linear space with the norm given by

\begin{equation}
\label{lipnorm}
\Vert f\Vert_{\Lip} = \max \bigl\{\sup_{x \in X} \bigl\vert f (x)\bigr\vert,  \Lip (f)\bigr\}.
\end{equation}

\noindent  The metric $d$ induces a linear transformation $J_{d}: L^2 (\mu) \aro \Lip \, (X)$ by setting\footnote{We note that the distinction between $T_{d}$ and $J_{d}$ is in the codomain.}
\[
\left(J_{d} f \right) (x) = \langle\iota_{d} (x), f\rangle,
\]

\noindent for $ f \in L^2 (\mu)$.  That $J_{d} f \in \Lip \, (X)$ follows from
\[
\aligned
\bigl\vert\left(J_{d} f\right) (x) - \left(J_{d} f\right) (y)\bigr\vert & \le \int\bigl\vert d (x, z) - d (y, z)\bigr\vert \bigl\vert f (z)\bigr\vert d \mu (z)  \\
& \le \mu (X)  \bigl\Vert f\bigr\Vert_2  d (x, y)   
\endaligned
\]

\noindent so that $\Lip\,  (J_d f) \le \mu (X) \bigl\Vert f \bigr\Vert_2$.

\begin{lemma}
\label{bnddlinear} 
Let   $(X, d, \mu)$ be metric-measure space with $\diam \, (X, d) < \infty$.    Then the linear map $J_{d}: L^2 (\mu) \aro \Lip \, (X)$ is  bounded, and hence, Lipschitz.
\end{lemma}

\begin{proof}   Let $f \in L^2 (\mu)$.  We have already seen  that $\Lip\, (J_d f) \le \mu (X) \bigl\Vert f \bigr\Vert_2$.  Moreover, 
\[ 
\bigl\vert \bigl(J_{d} f\bigr) (x)\bigr\vert \le \int d (x, z) \bigl\vert f (z)\bigr\vert d \mu (z) \le \diam \,  (X, d) \, \mu (X) \bigl\Vert f \bigr\Vert_2.
\]

\noindent  Consequently, the Lipschitz norm
\[
\bigl\Vert J_{d} f\bigr\Vert_{\Lip}  \le \max\bigl\{1, \diam \, (X, d)\bigr\} \mu (X)  \bigl\Vert f\bigr\Vert_2
\]

\noindent so that the operator norm $\Vert J_{d}\Vert \le   \max \bigl\{1, \diam \, (X, d)\bigr\} \mu (X)   <  \infty$.
\end{proof}

Since  the linear map $T_{d}$ is compact and   the inclusion $\xymatrix{C  (X) \ar@{^{(}->}[r] & L^2 (\mu)}$ is continuous,    the composite $\xymatrix{\Lip \, (X)  \ar@{^{(}->}[r] & C (X) \ar@{^{(}->}[r] & L^2 (\mu)}$ is compact, and we have the commutative diagram

\begin{equation}
\label{incl.compact}
\xymatrix{
 & \Lip \, (X) \ar@{^{(}->}[d] \\
L^2 (\mu) \ar[ru]^{J_{d}} \ar[r]^{T_{d}} & L^2 (\mu).
}
\end{equation}

\noindent It is  well known\footnote{See, for instance, \cite{bilip}.} that the evaluation map $\ev: X \aro \Lip^{\ast} (X)$ given by $\ev (x) (f) = f (x)$ is a bi-Lipschitz (in fact, isometric; see \cite{pestov}) embedding.    Letting $\eta: L^2 (\mu) \aro L^2 (X)^{\ast}$ denote the isomorphism $\eta (f) = \langle f, - \rangle$, we compute

\begin{equation}
\label{ev&eta}
\aligned
\bigl(\eta \circ \iota_{d}\bigr) (x) (f) & =  \langle \iota_{d} (x), f\rangle = \int d (x, z) f (z) d \mu (z) \\
& = \bigl(J_{d} f\bigr) (x) = \langle \ev (x), J_{d} f\rangle \\
& = \langle \bigl(J^{\ast}_{d} \circ \ev\bigr) (x), f\rangle = \bigl( J^{\ast}_{d} \circ \ev\bigr) (x) (f).
\endaligned
\end{equation}

\noindent Here, $J^{\ast}_{d}: \Lip^{\ast} (X) \aro L^2 (\mu)^{\ast}$ is the adjoint of $J_{d}$.  Since (\ref{ev&eta}) holds for every $x \in X$ and every $f \in L^2 (\mu)$, we have the commutative diagram

\begin{equation}
\label{CanEmbEv}
\xymatrix{
X \ar[r]^{\iota_{d}} \ar@{^{(}->}[d]_{\ev} & L^2 (\mu) \ar[d]^{\eta}\\
 \Lip^{\ast} (X) \ar[r]^{J_{d}^{\ast} }& L^2 (\mu)^{\ast}.
}
\end{equation}

 We  now recall the following discussion from \cite{bilip}.   Let $(X, d, \mu)$ be a compact metric-measure space. Then the canonical map $\iota_{d}: X \aro L^2 (\mu)$ lifts to $\lambda_{d}: X \aro \Lip \, (X)$ by setting $\lambda_{d} (x) = d (x, -)$.   This lift $\lambda_d$ is totally discontinuous with $\lambda_d (X)$ metrically discrete: For $x, y \in X$, we have
\[
\left\Vert \lambda_d (x) -  \lambda_d(y)\right\Vert_{\Lip} =  \max \bigl\{2 , d (x, y)\bigr\}.
\]

\begin{cor} \cite[Proposition 5.2]{bilip}
\label{unsep}
 Let $(X, d)$ be a compact  uncountable metric space.  Then the set $\lambda_d (X)$ is closed, uncountable and discrete in the norm topology of $\Lip\, (X)$.  In particular, $\Lip\, (X)$ is not separable.
\end{cor}

\noindent In contrast, however, it follows from simple calculation that   the lift  $\lambda_{\Delta (d)}: X \aro \Lip\, (X)$ of $\iota_{\Delta (d)}$ is Lipschitz. Then  diagram (\ref{k*k}) may be amended to give the commutative diagram

\begin{equation}
\label{liftcan}
\xymatrix{
& & L^2 (\mu) \ar[dd]^{J_{d}} \ar[dl]^{T_{d}} \\
X \ar@/^1.1pc/[urr]^{\iota_{d}} \ar[r]^{\iota_{\Delta (d)}} \ar@/_1.0pc/[drr]^{\lambda_{\Delta (d)}} & L^2 (\mu) \\
&& \Lip\, (X). \ar@{_{(}->}[ul]
}
\end{equation}

We next show that, when viewed as a  map of  certain  snowflakes of $(X, d)$,  the map $\lambda_{d}$ is actually Lipschitz.    Specifically, let  $0   <  s <  1$.    As usual, for $f \in \R^X$, we write 
$
\Lip_s (f) = \Lip \bigl(f: (X, d^s) \aro \R\bigr)
$
and 
$\Lip_s (X) = \left\{f \in \R^X \bigm\vert \Lip_s (f) < \infty\right\}
$.
Then $\Lip_s (X)$ is a normed linear space with $\Vert f\Vert_{\Lip_s} = \max \left\{\Vert f \Vert_{\infty}, \Lip_s (f)\right\}$. 

\begin{lemma}
\label{lipsnow}
Let $(X, d)$ be a metric space with  $\diam (X, d) < \infty$ and let $\lambda_{d} (x) = d (x, -)$.     Then, for $0 < s < 1$,  the map $\lambda_d : (X, d^s) \aro \Lip_{1 - s} (X)$   is Lipschitz with $\Lip (\lambda_d) \le  \max\bigl\{2,  \bigl(\diam (X, d)\bigr)^{1 - s}\bigr\}$.
\end{lemma}

\begin{proof}
Let $x \in X$.  Then 
\[
\Lip_{1 - s} \bigl(\lambda_d (x)\bigr) = \sup_{y \ne z} \frac{\bigl\vert \lambda_d (x) (y) - \lambda_d (x) (z)\bigr\vert}{d^{1 - s} (y, z)} \le d^s (y, z) \le \bigl(\diam (X, d)\bigr)^{s}
\]

\noindent so that, for each $x \in X$, the function $\lambda_d (x) \in \Lip_{1 - s} (X)$ with $\Lip_{1 - s} \bigl(\lambda_d (x)\bigr)  \le \bigl(\diam (X, d)\bigr)^{s}$.  Consequently, $\lambda_d : (X, d^s) \aro \Lip_{1 - s} (X)$, and we compute

\[
\aligned
\Lip (\lambda_d)  &= \sup_{x \ne y} \frac{\bigl\Vert \lambda_d (x) - \lambda_d (y)\bigr\Vert_{\Lip_{1 - s}}}{d^s (x, y)}\\
& =  \sup_{x \ne y} \frac{\max \left\{\bigl\Vert \lambda_d (x) - \lambda_d (y)\bigr\Vert_{\infty},  \Lip_{1 - s} \bigl(\lambda_d(x) - \lambda_d (y)\bigr) \right\}}{d^s (x, y)}.
\endaligned
\]

\noindent But $\bigl\Vert\lambda_d (x) - \lambda_d (y)\bigr\Vert_{\infty} \le  d (x, y)$ and 

\[
\aligned
\Lip_{1 - s} \bigl(\lambda_d(x) - \lambda_d (y)\bigr)  
&= \sup_{z \ne w}\frac{\bigl\vert\bigl(d (x, z) - d (y, z)\bigr) - \bigl(d (x, w) - d (y, w)\bigr)\bigr\vert}{d^{1 - s} (z, w)} \\
&\le \sup_{z \ne w} \frac{2 \min\bigl\{d (x, y), d (z, w)\bigr\}}{d^{1 - s} (z, w)} \le \sup_{z \ne w}\frac{2 d^s (x, y) d^{1 - s} (z, w)}{d^{1 - s} (z, w)}\\
&= 2 d^s (x, y).
\endaligned
\]

\noindent  Hence, we have

\[
\Lip (\lambda_d)  \le  \sup_{x \ne y} \frac{\max \bigl\{d (x, y), 2 d^s (x, y)\bigr\}}{d^s (x, y)} \le \max \left\{2, \bigl(\diam (X, d)\bigr)^{1 - s}\right\}.
\]
 
\end{proof}



\section{A bi-Lipschitz embedding result}

Let  $(X, d, \mu)$ be a  metric-measure space    and  let  $\iota_d: (X, d) \aro L^2 (\mu)$ be the canonical map.    In this section we prove the following result.

\begin{thm}
\label{iotarho} 
Let  $(X, d, \mu)$ be a compact metric-measure space and assume that the atoms of $\mu$ are isolated.  Let  $\iota_d: (X, d) \aro L^2 (\mu)$ be the canonical map and let $\rho_d  = \iota_d^{\ast} \Vert\cdot - \cdot\Vert_2$.  If the canonical map $\iota_{\rho_d}: (X, d) \aro L^2 (\mu)$ is lower Lipschitz,  then   there exists a bi-Lipschitz embedding  of $(X, d)$ into some $\R^N$.
\end{thm}

\noindent  It follows that if $\iota_d$ is an isometry, then there exists a bi-Lipschitz embedding  of $(X, d)$ into some $\R^N$.   However,   requiring the canonical map $\iota_d: (X, d) \aro L^2 (\mu)$  to be an isometry  is  very restrictive.  Indeed, if   $\iota_d$ is an isometry, then every triple of points in $X$ lie on a  geodesic.  Consequently,  $(X, d)$ is a  compact subset of the circle with the arc-length metric or a compact subset of the real line with the standard  metric.

\vskip5pt

  We will need the following definition from \cite{bilip} and references therein.    Let $X$ be a subset of a Banach space $\mathbb B$ and let 
\[
U (X) = \left\{\frac{x - y}{\Vert x - y\Vert_{\mathbb B}} \biggm\vert x, \, y \in X  {\text{ with }} x \ne y\right\}.
\]

\noindent  We say that $X$ is {\bf{spherically compact}} if $U (X)$ has a compact closure in the norm topology of $\mathbb B$. For compact subsets of a Banach space, spherical compactness is an invariant under $C^1$-diffeomorphisms but not under bi-Lipschitz equivalence.   We say that $X$ is {\bf{weakly spherically compact}} provided that the weak closure of $U (X)$ does not contain $0 \in \mathbb B$.  For compact subsets, weak spherical compactness is a bi-Lipschitz invariant.  Obviously, spherically compact implies weakly spherically compact, but not conversely.

\begin{lemma}  \cite[Lemma 2.4]{bilip}
\label{wscfin} 
Let $X$ be a compact subset of a normed linear space $\mathbb B$.  Then there exist  $N \in \N$ and a bounded linear map $\F: \mathbb B \aro \R^N$ with $\F\big\vert_X$ a bi-Lipschitz embedding if and only if $X$ is weakly spherically compact.
\end{lemma}

The next  result we need is the correct   version  of  \cite[Theorem 4.2]{bilip}.  The hypothesis  in  \cite[Theorem 4.2, (1)]{bilip} that $X$ is  weakly  spherically  compact   is not sufficient.  Instead, it must be assumed that $X$ is spherically compact.

\begin{thm}  \cite[Theorem 4.2]{bilip}
\label{hilwsc} 
Let $X$ be a compact subset of a Hilbert space $\h$ with the inherited metric $\Vert\cdot - \cdot\Vert_{\h}$ and let $\mu$ be a Borel regular measure with closed support $X$.
\begin{enumerate}
\item  If $X$ is spherically compact, then the canonical map $\iota_{\Vert\cdot - \cdot\Vert_{\h}}: X \aro L^2 (\mu)$ is bi-Lipschitz.
\item   Assume that all atoms of $\mu$ are isolated.   If the canonical map $\iota_{\Vert\cdot - \cdot\Vert_{\h}}:X \aro  L^2 (\mu)$ is bi-Lipschitz,  then $X$ is weakly spherically compact.
\end{enumerate}
\end{thm}

\noindent Part (2) of this theorem is needed  in  the proof  of Theorem  \ref{iotarho}.   In \cite{LangPlaut} it is asked whether every doubling subset of a Hilbert space admits a bi-Lipschitz embedding into some Euclidean space.   This problem is still open, but for doubling subsets of $L^p$, $p > 2$, the answer is known to be negative \cite{naor}.  Part (2) of Theorem \ref{hilwsc} above shows that, for compact subsets, if we replace the doubling condition with the condition that the subset    supports a   Borel regular measure  with isolated atoms and bi-Lipschitz canonical map, then  it does admit a bi-Lipschitz embedding into some  Euclidean space.  See also Proposition \ref{inhil} for another such result.

\begin{proof}[Proof of Theorem  \ref{iotarho}]
We define a Borel regular probability measure $\nu$ on $\iota_d (X)$ by setting $\nu = (\iota_d)_{\ast} \mu$.  That is,  $\nu (A) = \mu \bigl(\iota_d^{- 1} (A)\bigr)$ for every Borel subset $A \subset \iota_d (X)$.  Then, letting  $\zeta =  \iota_d^{\ast} \circ  \iota_{\Vert\cdot - \cdot\Vert_2}\big\vert_{\iota_d (X)}$,  the diagram

\[
\xymatrix{
 & &\iota_d (X) \subset L^2 (\mu) \ar[d]^{\zeta} \\
X \ar[urr]^{\iota_d}     \ar[rr]^{\iota_{\rho_d}}& &  \iota_{\rho_d} (X) \subset L^2 (\mu)
}
\]

\noindent is commutative by definition.   If $\iota_{\rho_d}$ is lower Lipschitz, then so are the maps  $\iota_d$ and $\zeta$.  Moreover, since all atoms of $\mu$ are isolated, all atoms of $\nu$ are also isolated.  Hence, by  part $(2)$ of Theorem  \ref{hilwsc}, the set $\iota_d (X) \subset L^2 (\mu)$  is weakly spherically compact.  Consequently, by Lemma  \ref{wscfin},  there is a bi-Lipschitz embedding   $\xymatrix{f: \bigl(\iota_d (X), \Vert \cdot - \cdot\Vert_2\bigr) \,  \ar[r] &\R^N}$  for some $N \in \N$.    Then the composite $\xymatrix{f\circ \iota_d: (X, d)     \ar[r] &\R^N}$  is a bi-Lipschitz embedding.
\end{proof}


\section{Another bi-Lipschitz embedding result}

The next result,  which  depends on Theorem \ref{iotarho},  states that the same conclusion holds provided that $\iota_d$ is not ``too far"   from an isometry.

\begin{thm}
\label{iotaneariso} 
Let  $(X, d, \mu)$ be a compact metric-measure space and assume that the  atoms of $\mu$ are isolated.   Assume further that the  canonical map $\iota_{d}: (X, d) \aro L^2 (\mu)$ is lower Lipschitz.     If the lower Lipschitz constant of $\iota_d$  is sufficiently  close -- depending on on the metric space $(X, d)$ -- to $1$, then   there exists a bi-Lipschitz embedding  of $(X, d)$ into some $\R^N$.
\end{thm}

Let $(X, d, \mu)$ be a metric-measure space  and let $\ep > 0$. For $x, y \in X$ we set

\begin{equation}
\label{Exy}
E (x, y, \ep; d) = \left\{z \in X \bigm\vert \bigl\vert d (x, z) - d (y, z)\bigr\vert \ge \ep \, d (x, y)\right\}.
\end{equation}

\noindent As in \cite{bilip}, we say that a   regular Borel measure $\mu$ on $X$ {\bf{separates points uniformly}} with respect to the metric $d$ provided that there exist $\ep_d > 0$ and $c_d > 0$ such that for every $x, y \in  X$ we have $\mu \bigl(E (x, y, \ep_d; d)\bigr) \ge c_d$.  We emphasize that the constants $\ep_d$ and $c_d$  are independent of $x$ and $y$ and depend only on $d$.

\vskip5pt

 For $(X, d, \mu)$ compact, the uniform separation  condition is equivalent to requiring the canonical map $\iota_d: X \aro L^2 (\mu)$   to be bi-Lipschitz.  

\begin{thm}  \cite[Theorem 4.1]{bilip}
\label{bL-sep} 
Let $(X, d, \mu)$ be a compact metric-measure space.  Then the measure  $\mu$ separates points uniformly with respect to $d$  if and only if   the canonical map $\iota_d: (X, d) \aro L^2 (\mu)$  is bi-Lipschitz.
\end{thm}

Let $\mathcal G (d)$ denote the class of all metrics on $X$ which are bi-Lipschitz equivalent to $d$.  Clearly, given $\sigma, \f \in \mathcal G (d)$,  there exist $0 < \ell (\sigma, \f) \le L (\sigma, \f) < \infty$ with      $\ell (\sigma, \f) \, \sigma \le \f \le L (\sigma, \f) \, \sigma$.           In the terminology of \cite{HeinonenSullivan}, $\mathcal G (d)$ is a metric gauge so that the bi-Lipschitz embedding problem is the question   whether this metric gauge contains a Euclidean metric.

\vskip5pt

We define a pseudometric $W_d$ on $\mathcal G (d)$ as follows:  Given $\sigma$ and $\f$ in $\mathcal G  (d)$, we let
\[
w (\sigma, \f) (x, y) =  \sup \left\{\vert (\sigma - \f) (x, z) - (\sigma - \f) (y, z)\vert \bigm\vert z \in X \right\},
\]

\noindent for $x, y \in X$, and then we set 

\begin{equation}
\label{metmet}
W_d (\sigma, \f) = \sup \left.\left\{\frac{w (\sigma, \f) (x, y)}{d (x, y)} \right\vert x \ne y \in X\right\}.
\end{equation}

\noindent That $W_d$ is a pseudometric follows from simple computation.  In detail,  to prove the triangle inequality, let $\sigma, \f, \kappa \in \mathcal G (d)$.  Then, for any $x, y \in X$, we have
\[
\aligned
w (\sigma, \f) (x, y) + w (\f, \kappa) (x, y) 
&\ge \sup \bigl\{\vert (\sigma - \f) (x, z) - (\sigma - \f) (y, z) \vert  \bigr.  \\
&{}\ \   \  \qquad + \bigl. \vert (\f - \kappa) (x, z) - (\f - \kappa) (y, z)\vert\bigm\vert z \in X \bigr\} \\
&\ge \sup \bigl\{\vert (\sigma - \kappa) (x, z) - (\sigma - \kappa) (y, z)\vert \bigm\vert z \in X \bigr\} \\
&= w (\sigma, \kappa) (x, y).
\endaligned
\]

\noindent  We note that if $\sigma \in \mathcal G (d)$, then for $0 \le s, t$ we have $W_d (s \sigma, t \sigma) \le \bigl\vert t - s\bigr\vert  L (d, \sigma)$.  Moreover, $\ell (d, \sigma) W_{\sigma}  \le W_d \le L (d, \sigma) W_{\sigma}$.

\vskip5pt

 Now let  $\mathcal E (d)$ denote the (possibly empty) subset of  $\mathcal G (d)$ consisting of the metrics whose corresponding canonical maps      $X \aro L^2 (\mu)$   are bi-Lipschitz.

\begin{lemma}
\label{strctsmll} 
Let $(X, d, \mu)$ be a compact  metric-measure space . Then  $\mathcal E (d)$ is open in $\bigl(\mathcal G (d),  W_d\bigr)$.
\end{lemma} 
 
\begin{proof} 
 If $\mathcal E (d)$ is empty, there is nothing to prove.  Otherwise,   let $\sigma \in \mathcal E (d)$.  Then  there exists $0 <\ell (d, \sigma)$ such that $\ell (d, \sigma) \, d \le \sigma$.   Moreover, by Theorem \ref{bL-sep}, there exist $\ep (d, \sigma) > 0$ and $c (d, \sigma) >  0$ such that $\mu \bigl(E ( x, y, \ep (d, \sigma); \sigma)\bigr) \ge c (d, \sigma)$ for all $x, y \in X$; see (\ref{Exy}).    Similarly, given  $\f \in \mathcal G (d)$, there exists $0 <\ell (\f, \sigma)$ such that $\ell (\f, \sigma) \, \f \le \sigma$ and we have 
\[
\aligned
\ep (d, \sigma) &\le \frac{\bigl\vert\sigma (x, z) - \sigma (y, z)\bigr\vert}{\sigma (x, y)} \\
& \le\frac{\bigl\vert\f (x, z) - \f (y, z)\bigr\vert}{\sigma (x, y)} + \frac{\bigl\vert\sigma (x, z) - \f (x, z) - \bigl(\sigma (y, z) - \f (y, z)\bigr)\bigr\vert}{\sigma (x, y)}\\
&\le \frac{\bigl\vert\f (x, z) - \f (y, z)\bigr\vert}{\ell (\f, \sigma) \, \f (x, y)}  + \frac{W_d (\sigma,  \f)}{\ell (d, \sigma)}.
\endaligned
\]

\noindent Hence, if $W_d (\sigma, \f) < \ell (d, \sigma) \, \ep (d, \sigma) /2$, then 
\[
0 <  \ell (\f, \sigma) \left(\frac{\ep (d, \sigma)}{2}\right) \le \frac{\bigl\vert\f (x, z) - \f (y, z)\bigr\vert}{ \f (x, y)},
\]

\noindent and $E ( x, y, \ep ( d, \sigma); \sigma) \subset E ( x, y, \ell (\f, \sigma) \, \ep (d, \sigma) / 2; \f)$.   Consequently, $\f \in \cate (d)$. 
\end{proof}

\begin{proof}[Proof  of Theorem \ref{iotaneariso}]   Since $\iota_d$ is bi-Lipschitz,    $d \in \cate (d)$.  Then by Lemma \ref{strctsmll} there is  $r_d > 0$ such that the  open  $W_d$-ball $B (d, r_d)$  of radius $r_d$ centered at $d$ is contained in $\cate (d)$.

\vskip5pt

Again, letting $\rho_d = \iota_d^{\ast} \Vert\cdot - \cdot\Vert_2$,  there is  
$0 < \ell \le 1$ with   $\ell\, d \le \rho_d \le d$ implying  that $d - \rho_d \le (1 - \ell) d$.   It follows that $(1 - \ell) / (1 + \ell) \le  W_d (d, \rho_d) \le (1 - \ell) / \ell$ so that if $(1 - \ell) / \ell < r_d$, equivalently $1 / (1 + r_d ) < \ell$, then $\rho_d \in \cate (d)$.  In this case     $\iota_{\rho_d}: (X, d) \aro L^2 (\mu)$ is lower Lipschitz and,   by Theorem \ref{iotarho},   there exists a bi-Lipschitz embedding  $\xymatrix{(X, d)  \ar[r] &\R^N}$  for some $N \in \N$.
\end{proof}


 \section{A third bi-Lipschitz embedding result}

Let     $(X, d, \mu)$  be a metric-measure space.  As in Section 2,  we use  the inner product in $L^2 (\mu)$  to define a  symmetric function $\Delta (d): X \times X \aro [0, \infty)$ by setting

\begin{equation}
\label{convol}
\Delta (d) (x, y)  = \langle d (x, -), \, d (-, y)\rangle = \langle \iota_d (x), \, \iota_d (y)\rangle.
\end{equation}

\noindent   Although $\Delta (d)$ is not a metric, it does   induce a canonical map $\iota_{\Delta (d)}: X \aro L^2 (\mu)$ by setting $\iota_{\Delta (d)} (x) = \Delta (d) (x, -)$.  Moreover, if $\diam\, (X, d) < \infty$, then the map $\iota_{\Delta (d)}$ is Lipschitz; see Lemma \ref{kstark} below.

\begin{thm}
\label{d*d} 
Let  $(X, d, \mu)$ be a compact metric-measure space.  If the canonical map $\iota_{\Delta (d)}: X \aro L^2 (\mu)$ is lower Lipschitz, then   there exists a bi-Lipschitz embedding  of $(X, d)$ into some $\R^N$.
\end{thm}

For the proof, we will need the following special case of \cite[Lemma 2.3]{bilip}.

\begin{lemma}  cf. \cite[Lemma 2.3]{bilip}
\label{complinear} 
Let $\F : \mathbb H_1 \aro \mathbb H_2$ be a compact linear map of Hilbert spaces and let  $X \subset \mathbb H_1$ be a compact subset.  If $\F\big\vert_X$  is bi-Lipschitz,  then  $\F (X)$  is spherically compact and  $X$   itself is weakly spherically compact.
\end{lemma}

Given   a metric-measure space $(X, d, \mu)$,     it follows from simple calculations that, for $x, y \in X$,  have 

\[\aligned
\bigl\vert\bigl(1 - d (x, y)\bigr) &\min \left\{\Vert\iota_{d} (x)\Vert^2_2,\, \Vert\iota_{d^{1/2}}(x)\Vert_2^2\right\}  \bigr\vert
 \le \Delta  (d) (x, y) \\
&\le \bigl(1 + d (x, y)\bigr) \max \left\{\Vert\iota_{d} (x)\Vert^2_2  ,\, \Vert\iota_{d^{1/2}}(x)\Vert_2^2\right\},
\endaligned\]

\noindent and similarly for $y$, and that

\begin {equation}
\label{iotadd}
\Vert\iota_{\Delta (d)} (x) - \iota_{\Delta (d)} (y)\Vert_2 \le \diam\,  (X, d) \,  \mu^{3/2} (X) \, d (x, y).
\end{equation} 

\begin{lemma}
\label{kstark}
Let $(X, d, \mu)$ be a  metric-measure space.   Then the canonical map $\iota_{\Delta (d)}: X \aro L^2 (\mu)$    is injective.  Furthermore, if $\diam\, (X, d) < \infty$, then $\iota_{\Delta(d)}$ is Lipschitz.
\end{lemma}  

\begin{proof}
Assume that there exist $x, y \in X$ such that $\iota_{\Delta (d)} (x) = \iota_{\Delta (d)} (y)$.  Then for $\mu$-a.e. $z \in X$ we have $\langle\iota_{d} (x) - \iota_{d} (y),  \iota_{d} (z)\rangle = 0$.  But $\iota_{d}$ is Lipschitz and therefore continuous.  Hence, $\langle\iota_{d} (x) - \iota_{d} (y),  \iota_{d} (z)\rangle = 0$ for all $z \in X$.  In particular, setting $z = x$ and then $z = y$ gives $0 = \langle\iota_{d} (x) - \iota_{d} (y),  \iota_{d} (x) - \iota_{d} (y)\rangle = \bigl\Vert\iota_{d} (x) - \iota_{d} (y)\bigr\Vert_2$.  Since $\iota_{d}: X \aro L^2 (\mu)$ is injective,  we have $x = y$.

The second part follows directly from inequality  (\ref{iotadd}).
\end{proof}

\begin{proof}[Proof of Theorem \ref{d*d}]  Let $(X, d, \mu)$ be a compact metric-measure space.  Then in  diagram (\ref{k*k}),
  the map $\iota_{\Delta (d)}$ is  one-one and Lipschitz   by Lemma \ref{kstark}.   Moreover,  since  $\iota_d$ is Lipschitz and $T_d$ is compact,   if $\iota_{\Delta (d)}$ is lower Lipschitz, then the linear map $T_d\bigl\vert_{\iota_d (X)}$ is bi-Lipschitz.  Hence,  the set $\iota_d (X) \subset L^2 (\mu)$ is weakly spherically compact by Lemma \ref{complinear}.   The  result now follows immediately from  Lemma \ref{wscfin}.  
\end{proof}


\section{Remarks}

\noindent {\bf I. }  The following Example is folklore; it was shown to us by Steve Armentrout in the year 1990.

\begin{examples} 
\label{noiso}
Let $X = \{z, x_1, x_2, x_3\}$  and,   for $i = 1, 2, 3$, let $d (x_i, x_i) = 0$,  $d (x_i, x_j) = 2$ for $i \ne j$,  and   $d (z, x_i) = 1$.  It is easily checked that there is no isometric embedding of $(X, d)$ into $\R^N$, any $N \ge 1$.  
\end{examples}

\vskip6pt

\noindent {\bf II. } Let $(X, d)$ be a metric space and  let  $0 < t  <  s \le 1$.  If  the identitiy map $\id : (X, d^s) \aro (X, d^t)$ is Lipschitz, then the space $(X, d)$ is descrete.  However, if $\diam \, (X, d) < \infty$, then  the identitiy map $\id : (X, d^t) \aro (X, d^s)$ is Lipschitz with $\Lip (\id) \le \bigl(\diam (X, d)\bigr)^{s - t}$.

\vskip5pt

Now let $1 \le p < \infty$ and let $\iota (s)$ be the composite $\xymatrix{ (X, d^s) \ar[r]^{\id}  & (X, d) \ar[r]^{\iota_d} & L^p (\mu)}$ so that $\iota (s) (x) = d (x, -)$.  It is clear that   the map $\iota (s)$ is Lipschitz. However, if $\iota (s)$ is Lipshitz below, then the map $\id : (X, d) \aro (X, d^s)$ is Lipschitz.

\begin{cor}
\label{iotas}
Let $(X, d, \mu)$ be a metric-measure  space with $\diam (X, d) < \infty$,  let $0 < s < 1$, and let $1 \le p < \infty$.  If the map  $\iota (s) : (X, d^s) \aro L^p (\mu)$, given  by setting $\iota (s) (x) = d (x, -)$,  is bi-Lipschitz, then the space $(X, d)$ is descrete.  
\end{cor}

\vskip6pt

\noindent {\bf III. }  Let $(X, d_1, \mu)$ be a metric-measure space and let $d_2$ be another metric on $X$ satisfying $\ell d_1 \le d_2 \le d_1$ for some $0 < \ell \le 1$.  Assume that the measure $\mu$ separates points uniformly with respect to $d_1$ and let $\delta = d_1 - d_2$.  If there exists $k \ge 0$ such that $\bigl\vert\delta (x, z) - \delta (y,  z)\bigr\vert  \le k \delta (x, y)$, for all $x, y, z \in X$, then $E (x, y, \ep, d_1)  \subset E (x, y, \ep /k, d_2)$ so that  $\mu$ separates points uniformly with respect to $d_2$ as well.

\vskip6pt

\noindent {\bf IV. } Let $(X, d, \mu)$ be a metric-measure space  and consider the following two concepts:

\begin{enumerate}
\item  Let $\underline{D}: X \aro [0, \infty) \cup \{\infty\}$ by setting 
\[
\underline{D} (x) = \liminf_{r \to 0} \frac{\log \mu \bigl(B (x, r)\bigr)}{\log r},
\]

\noindent where $B(x, r)$ denotes the closed ball of radius $r \ge 0$ centered at the point $x \in X$.  In \cite{Lie}, the number $\underline{D} (x)$ is called the {\bf{lower mass-scaling dimension}} of $(X, d, \mu)$ at the point $x \in X$.  By taking $\limsup$ instead of $\liminf$ we get $\overline{D} (x)$, the  {\bf upper mass-scaling dimension} at the point $x$.  It is a classical result \cite{bill}  that if there exist $0 \le \ep_1 \le \ep_2$ with  $\ep_1 \le \underline{D} (x) \le \overline{D} (x) \le \ep_2$, for all $x \in X$, then the Hausdorff dimension of $(X, d)$ is in the interval $[\ep_1, \ep_2]$.

\item  The following condition, which we call {\bf{Ka\l-doubling}} condition,  is from \cite{kala}:  There is a function $C (r)$ satisfying  $\lim_{r \to 0} r^p \, C (r) = 0$ and every  $B (x, r)$ can be covered by at most $C (r)$ balls of radius $r/2$ and centers in $B (x, r)$.   Letting $W^{1, p} (X, d, \mu)$ denote the Sobolev space defined in \cite{haj}\footnote{Instead of  $W^{1, p} (X, d, \mu)$, the notation   $M^{1, p} (\mu)$ is used in \cite{haj}.}, it is shown in \cite{kala} that if $(X, d)$ satisfies the Ka\l-doubling condition, then the embedding  $\Psi: W^{1, p} (X, d, \mu)  \aro L^p (\mu)$ is compact.
\end{enumerate}

Now let $X$ be a compact subset of a Hilbert space $\h$ with the inherited metric and let $\mu$ be a finite  Borel regular measure with closed  support  $X$.  We assume that $(X, d = \Vert\cdot - \cdot\Vert_{\h})$  is Ka\l-doubling   and that  $2 < \inf_{x \in X} \underline{D} (x) < \infty$.\footnote{By taking the Cartesian product with the $3$-sphere $S^3$, if necessary, the lower inequality can always be guaranteed.}   As in  \cite{Lie},  we may then use the standard $C^1$ structure on $\h$ to show that

\begin{equation}
\label{sob1}
\bigg\vert\Vert x - u\Vert_{\h} - \Vert y - u\Vert_{\h} - \bigl[\Vert x - v\Vert_{\h} - \Vert y - v\Vert_{\h}\bigr]\bigg\vert  \le L \, \Vert x - y\Vert_{\h} \Vert u - v\Vert_{\h},
\end{equation}

\noindent for some $0 \le L < \infty$.  It follows that the canonical map $\iota_d\big\vert_X: (X,  \Vert\cdot - \cdot\Vert_{\h}) \aro L^2 (\mu)$ may be lifted uniquely to a Lipschitz map\footnote{Again, the distinction between $\iota_d$ and $\widehat\iota$ is in the codomain.} $\widehat\iota: (X,  \Vert\cdot - \cdot\Vert_{\h}) \aro  W^{1, 2} (X, \Vert\cdot - \cdot\Vert_{\h}, \mu)$  so that we have the commutative diagram  
\[
\label{sobol}
\xymatrix{  & W^{1, 2} (X, \Vert\cdot - \cdot\Vert_{\h}, \mu) \ar[d]^{\Psi}\\
\h \supset X \ar[ur]^{\widehat\iota} \ar[r]^{\iota_d} & L^2 (\mu)
}
\]

\noindent with $X$ compact, $\widehat\iota$ Lipschitz, and $\Psi$ a compact  linear map.  Hence, if $\iota_d\big\vert_X$ is bi-Lipschitz, then by \cite[Theorem 2.5]{bilip},  $\iota_d (X)$, and hence $X$ itself,  admits a bi-Lipschitz embedding into some Euclidean space.

\begin{prop} 
\label{inhil}
Let $X$ be a compact subset of a Hilbert space $\h$ with the inherited metric and let $\mu$ be a finite  Borel regular measure with closed  support  $X$.  We assume that $(X, d = \Vert\cdot - \cdot\Vert_{\h})$  is Ka\l-doubling, and we let $\iota_d\big\vert_X: (X,  \Vert\cdot - \cdot\Vert_{\h}) \aro L^2 (\mu)$ be the canonical map.  If  $\iota_d\big\vert_X$ is bi-Lipschitz, then    $X$  admits a bi-Lipschitz embedding into some $\R^N$.
\end{prop}

Unfortunately, we cannot prove (\ref{sob1}) for a general metric-measure space $(X, d,  \mu)$, but only 

\begin{equation}
\label{sob2}
\bigg\vert d ( x, u) - d (y,  u) - \bigl[d (x,  v)  - d (y,  v)\bigr]\bigg\vert  \le 2 \, d (x, y)^{1/p}  d (u,  v)^{1/q},
\end{equation}
  
\noindent for $1 < p, q < \infty$  and $1/p + 1/q = 1$.   From (\ref{sob2}) we obtain a Lipschitz lift $\hat\iota_d: (X, d^{1/p})  \aro W^{1, 2} (X, d^{1/2}, \mu)$ of $\iota_d: (X, d^{1/p}) \aro L^2 (\mu)$.  We now have the following result  which is akin to the bi-Lipachitz embedding theorem of Assouad  \cite{assouad2} for doubling metric spaces.

\begin{prop}
Let $(X, d, \mu)$ be a Ka\l-doubling compact metric-measure space.  If the map $\iota_d: (X, d^{1/p}) \aro L^2 (\mu)$ is bi-Lipschitz for some $1 < p < \infty$, then there exists a bi-Lipschitz embedding of $(X, d^{1/p})$ into some $\R^N$.
\end{prop}

\vskip6pt

\noindent {\bf V. }  For the sake of concreteness, it follows from elementary computation that for $X = [0, 1]$ with $d$ the standard metric and $\mu$ the Lebegue measure, the symmetric function $\Delta (d)$ is given by 

\[
\Delta(d) (x, y) = \left\{\aligned
&y^2 ( x - \frac{y}{3})  + x^2 (\frac{x}{3} - y) + x y - \frac{x + y}{2} + \frac{1}{3}, \ {\text{ if }} \,  0 \le y \le x \le 1,\\
&x^2 (y - \frac{x}{3})  + y^2 (\frac{y}{3} - x)  + x y - \frac{x + y}{2} + \frac{1}{3}, \ {\text{ if }} \,  0 \le x \le y \le 1.
\endaligned\right.
\]

\vskip6pt

\noindent {\bf VI. } Classically  (see, for instance,  \cite{dugundji, gromov}),  the canoncal map $\iota_d$ is viewed as a map $X \aro C (X) \subset L^{\infty} (\mu)$.  In this case, $\iota_d$ is an isometric embedding, but nowhere $C^1$ even when $X$ is a compact subset of $\R^n$. For example, for the unit circle $S^1$ with arclength metric and Lebesgue measure, $\iota_d (S^1) \subset C (S^1)$ has a ``corner'' at every point \cite{Lie}.

\vskip6pt

\noindent  {\bf{VII.}}   Let $d$ be  a metric on a set $X$.    In the language of enriched categories, the map $x \mapsto d (x, -)$ is  just the Yoneda embeding.  In order to be more specific,  we briefly recall  the Lawvere category \cite{lawvere}.  See \cite{maclane} for basic concepts of  categories, and \cite{Max} for those of enriched categories.

 \vskip5pt
 
 Let ${\mathcal L}_0$ denote the category whose objects are $[0, \infty) \cup \{\infty\}$ and whose maps are the reverse order.  That is,  $a \aro b$ if and only if $a \ge b$.  The Lawvere category $\mathcal L = \left(\mathcal L_0, +, 0\right)$ is a symmetric monoidal closed category with $ + $ as the monoidal operation,  $ 0 $ as the unit object, and   the cotensor $\map (a, b) =\max \{b - a,  0\}$ as the internal hom objects.  The twist isomorphism 
 $\map \bigl(a, \map (b, c)\bigr)  \cong \map \bigl(b, \map (a, c)\bigr)$
 is  just the identity\footnote{Both sides are  equal to $\max \{c - b - a, 0\}$.}.   An $\mathcal L$-category (also called an $\mathcal L$-enriched category) is a set  $X$ equipped with a   metric $d$.   Given two $\catl$-categories $(X, d)$ and $(Y, \delta)$, an $\catl$-functor $f: (X, d) \aro (Y, \delta)$ is just a contraction:  $\delta \bigl(f (x), f (x^{\prime})\bigr)  \le d (x, x^{\prime})$.   Hence, $f$ is an $\catl$-isomorphism if and only if it is an isometry.  An $\catl$-natural transformation from the $\catl$-functor $f: (X, d) \aro (Y, \delta)$ to the $\catl$-functor $g: (X, d) \aro (Y, \delta)$  consists of a family of maps $\eta = \{\eta_x\}_{x \in X}$ with $\eta_x: 0 \aro \delta \bigl(f (x), g (x)\bigr)$.  This, of course, means that there is at most one $\catl$-natural transformation from $f$ to $g$, and that is precisely when $\delta \bigl(f (x), g (x)\bigr) = 0$ for all $x \in X$.  Following \cite[page 29]{Max}, we write $\bigl[(X, d), (Y, \delta)\bigr]$ for the $\catl$-category whose underlying  ordinary category has objects the $\catl$-functors  $(X, d) \aro (Y, \delta)$ and maps the $\catl$-natural transformations between them.

 \vskip5pt
 
 Now, $\catl$ equipped with the family $\bigl\{\map (a, b)\bigm\vert a, b\in \catl_0\bigr\}$ as the hom objects is itself an $\catl$-category; see \cite[page 15]{Max}.  Hence, an $\catl$-functor $F: (X, d) \aro \catl$ satisfies $d (x, x^{\prime}) \ge \map \bigl( F (x), F (x^{\prime})\bigr) = \max \bigl\{F (x^{\prime}) - F (x), 0\bigr\}$.  In particular,  for each $x \in X$,  the functor $d (x, -): (X, d) \aro \catl$ is an $\catl$-functor  because $d$  satisfies  $d (y, z) \ge  \max \bigl\{d (x, z) - d (x, y), 0\bigr\}  = \map \bigl(d (x, y), d (x, z)\bigr)$.  The Yoneda embedding is the fully faithful $\catl$-functor $\mathbb Y: (X, d) \aro \bigl[(X, d), \catl\bigr]$ sending  the object $x$ to the representable $\catl$-functor  $d (x, -)$. 

 \vskip6pt

\noindent{\bf VIII.}     Either one  of our hypotheses

\begin{enumerate}
\item that $\iota_{\rho_d} : (X, d)  \aro L^2 (\mu)$ is bi-Lipschitz in Theorem \ref{iotarho}, or 
\item that $\iota_{\Delta (d)} : (X, d)  \aro L^2 (\mu)$ is bi-Lipschitz in Theorem \ref{d*d},
\end{enumerate}

\noindent implies that $\iota_d : (X, d)  \aro L^2 (\mu)$ is bi-Lipschitz.  In fact, writing $\ell (f)$ for the lower Lipschitz constant of a map $f$, we have $
\ell \left(\iota_{\Delta (d)}\right)  \inf_x \bigl\Vert\iota_d (x)\bigr\Vert_2   \le \ell \left(\iota_d\right)$.

\vskip6pt

\noindent{\bf IX.} In our  bi-Lipschitz embedding theorems above, that $(X, d)$ is a doubling (equivalently, of finite Assouad dimension \cite{assouad1}) metric space is a conclusion rather than a hypothesis.    As pointed out by  Eriksson-Bique \cite{sylvester}, perhaps   Statement A.1  (see the appendix below)  is true with some implied doubling hypothesis.


\vskip15pt

\hskip80pt  {\bf Appendix:  Erratum for  \cite{bilip}.}

\vskip8pt

This    erratum for \cite{bilip}  was scheduled to appear in Portugaliae Mathematica in 2020, but   somehow (most likely due to the COVID-19 pandemic)  it fell through the cracks and didn't happen.  So, here it is.  

\vskip5pt

The main ``result'' of  \cite{bilip} is the following statement:

\vskip6pt

\noindent {\bf{Statement A.1.}} \cite[Main Theorem]{bilip}  {\it{Let $(X, d, \mu)$ be a compact metric-measure space.  If the canonical map $\iota_d: (X, d) \aro L^p (\mu)$ is bi-Lipschitz for some $1 < p < \infty$, then there exists a bi-Lipschitz embedding of $(X, d)$ into some $\R^N$.}}

\vskip6pt

\noindent The above  remains simply a  ``statement''  because  Sylvester Eriksson-Bique \cite{sylvester},  and independently Basok and Zolotov \cite{mish-vlad},  found a fatal error in the proof presented in \cite[Lemma 3.6]{bilip}.  Specifically, in the proof of \cite[Lemma 3.6]{bilip}, it is falsely stated that ``If the sequence $\{i (n) \}_{n \ge 1}$ is not co-final in $I$, then the directed set $I^{\prime} = \{j : j > i (n), n \ge 1\}$ is non-empty and co-final in $I$."  Worse yet, they also constructed counterexamples to show that, as written,   Statement A.1  is indeed false.    The following is another such example; it is due an anonymous  reader.

\vskip5pt

\noindent{\bf Example A.2.} [Anonymous Reader]  Let $\lambda$ denote the Lebesgue measure on the interval $[0, 1]$. The first step is to find a compact non-doubling subset $L \subset L^{\infty} (\lambda) \subset L^2 (\lambda)$ so that  identity map ${\text{id}} : \bigl(L, \Vert\cdot\Vert_{\infty}\bigr) \aro \bigl(L, \Vert\cdot\Vert_2\bigr)$ is bi-Lipschitz  and $\Vert f\Vert_{\infty} \le M$ for all $f \in L$ and some fixed $M \ge 2$.

\vskip5pt

To this end, let $n \in \N$ and let $F_n = \{ -1, 1\}^{2^n}$ with the Hamming distance $d_n$:  For $\alpha = (a_1, \cdots, a_{2^n})$ and $\beta = (b_1, \cdots, b_{2^n})$ in $F_n$, we set $d_n (\alpha, \beta)$ equal to the number of indices at which $\alpha$ and $\beta$ differ.  Let $K_n \subset F_n$ be a Hadamard code: $K_n$ is a $2^{n - 1}$-separated subset of cardinality $2^n$.  In other words, for $\alpha, \beta \in K_n$ with $\alpha \ne \beta$ we have $d_n (\alpha, \beta) \ge 2^{n - 1}$.  Let $\nu_n$ denote the counting measure on $K_n$ normalized so that $\nu_n (K_n) = 1/2^n$.  Then   $\sum_{n \ge 1} \nu_n (K_n) = 1$, and each $\alpha \in K_n$ has measure $\nu_n (\alpha) =  1/2^{2 n}$.

\vskip5pt

Now let $h : \R \aro \R$ be the standard hat function centered at $t = 1/2$ and supported on the interval $[0, 1]$.  Explicitly, $h (t) = \max \left\{0, 1/2 - \vert t - 1/2\vert\right\}$, $t \in \R$.  Clearly, $h$ is Lipschitz with $\Lip (h) = 1$.  For each $n \ge 1$ and each $\alpha =  (a_1, \cdots, a_{2^n}) \in K_n$ we define a function $f_{n, \alpha} : \R \aro \R$ by setting 

\[
f_{n, \alpha} (t) = \frac{1}{2^n} \sum_{i = 0}^{2^n - 1} a_{i + 1} \, h (2^n t - i).
\]

\noindent Then $f_{n, \alpha}$ is a Lipschitz function supported on the interval $[0, 1]$ with $\Lip (f_{n, \alpha}) = 1$.  Consequently, each $f_{n, \alpha} \in L^{\infty} (\lambda) \subset L^2 (\lambda)$ with $\Vert f_{n, \alpha}\Vert_{\infty} = 1/ 2^{n + 1}$ while $\Vert f_{n, \alpha}\Vert_2 = 1 / (\sqrt{3}\, 2^{n + 1})$. 

\vskip5pt

We set 

\[
L = \bigcup_{n \ge 1} \bigcup_{\alpha \in K_n} \{f_{n, \alpha}\} \bigcup \{\bold 0\},
\]

\noindent where $\bold 0$ denotes the identically zero function on $[0, 1]$.   Then  $L \subset L^{\infty} (\lambda) \subset L^2 (\lambda)$ is compact and countable.   Moreover,  for $m \le n$ and $\alpha \ne \beta$,  we have  both $\Vert f_{n, \alpha} - f_{m, \beta}\Vert_{\infty} \sim 1/ 2^n$ as well as $\Vert f_{n, \alpha} - f_{m, \beta}\Vert_2 \sim 1 / 2^n$ so that for any $c \ge 2 \sqrt 3$ we have 

\begin{equation}
\label{inftwo}
\Vert f_{n, \alpha} - f_{m, \beta}\Vert_2 \le \Vert f_{n, \alpha} - f_{m, \beta}\Vert_{\infty} \le c \, \Vert f_{n, \alpha} - f_{m, \beta}\Vert_2
\end{equation}

\noindent implying that  ${\text{id}} : \bigl(L, \Vert\cdot\Vert_{\infty}\bigr) \aro \bigl(L, \Vert\cdot\Vert_2\bigr)$ is bi-Lipschitz.  The metric space $\bigl(L, \Vert\cdot\Vert_{\infty}\bigr)$ is non-doubling because  the ball of radius $1 / 2^n$ and centered at $\bold  0$ contains $2^n$ points with pairwise distances $\sim 1 / 2^n$ (the functions corresponding to $K_n \simeq \{f_{n, \alpha} \bigm\vert \alpha \in K_n\}$).  Lastly, we endow $L$ with a Borel regular probability measure $\nu$ by setting $\nu (f_{n, \alpha}) = 1 / 2^{2 n}$ and $\nu (\bold 0) = 0$.

\vskip5pt

Finally,  let $M \ge 2$ be fixed and let $X =  [0, 1] \cup L$.  We define a metric $d$ on $X$ by setting

\[
\aligned
d (t, s) & = \big\vert t - s\big\vert \  {\text{ for }} t, s \in [0, 1],\\
d (f, g) &= \Vert f - g\Vert_{\infty} \   {\text{ for }} f, g \in L, \\
{\text{and }} d (t, f) & = 2 M + f (t) \  {\text{ for }} t \in [0, 1] {\text{ and }} f \in L.
\endaligned
\]

\noindent In order to show that $d$ satisfies the triangle inequality, the inequalities (\ref{inftwo}) are needed in the mixed case.  As for a measure $\mu$, we set $\mu = \nu / 2 + \lambda / 2$.  Then $(X, d, \mu)$ is a compact  metric-measure space with $\mu (X) = 1$ which does not admit a bi-Lipschitz embedding into $\R^N$ for any $N \in \N$ because it is non-doubling.  However, it follows from straightforward calculations  (considering different cases separately, and again using (\ref{inftwo}) for the mixed case) that the canonical map $\iota_d : (X, d) \aro L^2 (\mu)$ is bi-Lipschitz.


\end{document}